\documentclass[11pt]{article}
\usepackage[utf8]{inputenc}
\usepackage[hidelinks]{hyperref}

\usepackage[a4paper, total={6in, 9in}]{geometry}
\usepackage{graphicx}
\usepackage{booktabs} 
\usepackage{array} 
\usepackage{paralist} 
\usepackage{amsmath}
\usepackage{amssymb}
\usepackage{amsthm}
\usepackage{xcolor}
\usepackage{xurl}
\usepackage{mathtools}
\usepackage{subfigure}
\usepackage{float}

\usepackage{multicol}

\newtheorem{theorem}{Theorem}[section]
\newtheorem*{acknowledgement*}{Acknowledgement}

\newtheorem{lemma}[theorem]{Lemma}

\newtheorem{proposition}[theorem]{Proposition}
\newtheorem{remark}[theorem]{Remark}

\def\R{\mathbb{R}}

\newcommand{\ep}{\epsilon}
\def\P{{\mathbb P}}

\DeclareMathAlphabet{\mathmybb}{U}{bbold}{m}{n}

\DeclareMathOperator{\E}{\mathbb E}

\title{The relative efficiency of sequential tests}
\author{Henri Doerks\footnote{E-mail: \url{henri.doerks@math.uu.se}}\\Department of Mathematics, Uppsala University\\ \\
Erik Ekström\footnote{E-mail: \url{ekstrom@math.uu.se}}\\Department of Mathematics, Uppsala University\\
\\Yuqiong Wang\footnote{E-mail: \url{yuqw@umich.edu}}\\Department of Mathematics, University of Michigan\\ \\
}

\date{\today}

\begin{document}
\maketitle

\begin{abstract}
While many statistical procedures rely on a fixed sample size, sequential methods allow 
a decision-maker to adapt the sample size to achieve a given precision.
In this way, sequential tests reduce the average number of observations required to achieve a given power of the test -- but by how much? To address this question, we focus on the scenario of testing the unknown drift of a Brownian motion, comparing the Wald sequential probability ratio test with tests that use a pre-determined fixed sample size. We provide precise bounds on the average reduction in sample size needed to achieve a desired precision. Specifically, we demonstrate that for symmetric error bounds, the sequential test reduces the average sample size by at least 36\% and by at most 75\%. Moreover, the reduction in sample size increases monotonically with the power of the test, meaning that the relative advantage of using a sequential test over a fixed sample size test grows as higher power is required. 
We also 
study the relative efficiency in the case with asymmetric error bounds, and we provide a lower bound in terms of the symmetric case.
\end{abstract}

\section{Introduction}

In many statistical applications, the conventional approach involves collecting a fixed, pre-determined number of observations before performing a hypothesis test. 
Sequential analysis offers a compelling alternative, allowing for decisions to be made dynamically as data is gathered. Potentially, this can reduce the number of observations needed to achieve a specified level of power, as the test may terminate early once sufficient evidence has been accumulated.
In view of this, a central question arises:
how much does the average sample size decrease when switching from a fixed-sample design to a sequential testing procedure? 

In this article, we investigate the classical hypothesis testing problem for a Brownian motion with an unknown drift that can take two possible values.  
Within this setting, we compare the Wald sequential probability ratio test (SPRT) to a traditional test based on a pre-determined fixed sample size. In particular, we aim to quantify the {\bf relative efficiency} of SPRT, measured as
the quotient of the expected sample size in sequential testing with the sample size in the
corresponding fixed-sample procedure.
In many discrete time statistical tests on i.i.d. samples, the log-likelihood is a random walk with unknown drift, and under mild moment
assumptions, their rescaled paths can be seen as a Brownian motion with unknown drift. Our study can thus be
read as a limiting reference point for such tests.

The SPRT was originally proposed by Wald in the 1940's (see \cite{W}), and its relative efficiency 
was soon studied by several authors, cf. \cite{A}, \cite{Be}, \cite{B} and \cite{EG}. In particular, 
it is well-known (see \cite{A}, \cite{EG}, \cite[Chapter 4.2]{S}) that the relative efficiency approaches $\frac{1}{4}$ (which corresponds to a 75\% reduction in the average sample size) as the power of the test approaches 1. 
Moreover, it has been noted (cf. \cite{A} and \cite{EG}), using numerical calculations, 
that the convergence to $\frac{1}{4}$ is rather slow, and that typical average sample size reductions range between 50\% and 60\%.

In the current article, we complement and reinforce the results of \cite{A}, \cite{Be}, \cite{B} and \cite{EG} using a theoretical approach. 
Our main result (Theorem~\ref{main}) shows that in the case with symmetric error bounds for the two hypotheses, the possible range of the relative efficiency of SPRT is between 
$\frac{1}{4}$ and $\frac{2}{\pi}\approx 0.64$. In this way, SPRT reduces the average sample size by no less than 36\%, and by 
at most 75\%.
Moreover, the relative efficiency of the SPRT decreases {\bf monotonically} with the required power of the test, which corresponds to a monotone increase in the average reduction of the sample size. 
In other words, higher precision levels lead to larger (relative) time savings when using a sequential test instead of a test with a fixed sample size. Additionally, we confirm the slow convergence as the power increases to 1 by providing the exact convergence rate, which is of logarithmic type.

For {\bf asymmetrically} specified error bounds, the lower bound of $\frac{1}{4}$ for the relative efficiency is still valid (cf. Theorem~\ref{mainasym}). 
Interestingly, however, there is no non-trivial upper bound for the relative efficiency, which can be arbitrarily close to 1; furthermore, this happens when the asymmetric error bounds are very small.

The remainder of the paper is organized as follows. In Section 2, we outline the framework of tests of an unknown drift of a Brownian motion, 
and present both the optimal test with a fixed sample size and the SPRT. In Section~\ref{sec3} we state and prove our main result, which shows that in our specific setting
the reduction of the sample size is increasing in the power, and with a range between 36\% and 75\%.
Finally, in Section~\ref{sec4} we study the asymmetric situation, allowing for different probabilities for the two error types.

\section{Testing an unknown drift of a Brownian motion}
\label{sec2}

We adopt a classical set-up of hypothesis testing of an unknown drift of a Brownian motion as follows. 
Let $\theta$ be an unknown parameter, and let an observation process $X$ be given by
\[X_t=\theta  t+ W_t,\]
where $W$ is a standard Brownian motion with $W_0=0$.
In the hypothesis testing problem, the two hypotheses $H_1$ and $H_{-1}$ are tested against each other, where
\[H_1:\theta=1\]
and 
\[H_{-1}:\theta=-1.\]
Here, in a test with a fixed sample size, one observes the process $X$ over a time-interval $[0,T]$, where $T$ is a constant chosen before the sample is collected. 
On the other hand, in a sequential test the observer may choose a stopping time $\tau$ 
to end the test once sufficient evidence for one of the hypotheses has been obtained.
Also, a decision rule is an $\mathcal F^X_{\tau}$-measurable (or $\mathcal F^X_{T}$-measurable in the case of fixed sample size)
random variable $d$ taking values in $\{-1,1\}$. 
The terminal decision $d$ indicates which hypothesis to accept as follows: if $d=1$ then $H_1$ is accepted,
and if $d=-1$ then $H_{-1}$ is accepted. 

Naturally, the higher precision one wants, the larger sample size should be expected.
Denote $1-\alpha\in(\frac{1}{2},1)$ the power of a test, so that $\alpha\in(0,\frac{1}{2})$ is the maximal probability with which the wrong hypothesis is accepted. For a decision pair $(\tau,d)$ to have power $1-\alpha$, we thus require that 
\[\P_{1}(d=-1)\leq \alpha \quad\quad\& \quad\quad \P_{-1}(d=1)\leq \alpha,\]
where $\mathbb P_a$, $a\in\{-1,1\}$, denotes a probability measure under which $\theta=a$.
Since every deterministic time $T$ is a stopping time, the average sample size $\E_{a}[\tau]$ in an optimal sequential test is necessarily smaller than the fixed-size sample $T$ needed to obtain the same precision due to the additional flexibility. 
The aim of the present paper is to quantify this improvement by studying the relative efficiency (of the form $\frac{\E_{a}[\tau]}{T}$), or equivalently the average reduction (of the form $\frac{T-\E_{a}[\tau]}{T}$) as a function of the power of the test. \looseness=-1
\subsection{Tests with a fixed sample size}

It is clear by symmetry that an optimal decision at a deterministic time $T$ is given by
\[d=1_{\{X_T\geq 0\}}-1_{\{X_T< 0\}}.\]
Therefore, to bound the error probabilities by $\alpha$, one needs to use a deterministic time $T\in(0,\infty)$ such that $T\geq T_\alpha$, with $T_\alpha$ defined by
\[\mathbb P_{-1}(X_{T_\alpha}\geq 0)=\Phi\left(-\sqrt{T_\alpha}\right)=\alpha,\]
where
\[\Phi(x)=\int_{-\infty}^x\varphi(y)\,dy\qquad\&\qquad\varphi(y)=\frac{1}{\sqrt{2\pi}}e^{-y^2/2}\]
are the distribution and density function of the standard normal distribution.
Consequently, 
\[T_\alpha=\left(\Phi^{-1}(\alpha)\right)^2.\]

\subsection{Sequential tests}

From classical results by Wald (for a thorough account of SPRT in the setting of a Brownian motion with unknown drift, see \cite[Chapter 4]{S}), there exists one test that minimizes simultaneously $\E_{-1}[\tau]$ and $\E_1[\tau]$ over all sequential tests with power $1-\alpha$. The stopping time in that test is given by
\[\tau_{\alpha}:=\inf\left\{t\geq 0:X_t\notin\left(\frac{1}{2}\ln\frac{\alpha}{1-\alpha}, \frac{1}{2}\ln\frac{1-\alpha}{\alpha}\right)\right\},\]
with a decision variable
\[d=1_{\{X_{\tau_\alpha} = \frac{1}{2}\ln\frac{1-\alpha}{\alpha}\}}- 
1_{\{X_{\tau_\alpha} = \frac{1}{2}\ln\frac{\alpha}{1-\alpha}\}}.\]
Moreover, the average sample size satisfies
\[\E_{-1}\left[\tau_{\alpha}\right] = \E_1\left[\tau_{\alpha}\right] 
= \eta(\alpha),\]
where
\[\eta(\alpha):=
\frac{\left(1-2\alpha\right)}{2}
\ln\frac{1-\alpha}{\alpha}.\]

\subsection{Dependence on the signal-to-noise ratio}

Above, we are equipped with an observation process 
$X_t=\theta t + W_t$ and we test the hypotheses $H_1: \theta=1$ and $H_{-1}:\theta=-1$ against each other, which corresponds to a constant signal-to-noise ratio $2$. If the observation process instead is given by 
\begin{equation}
    \label{Xmu}
    X_t=\theta \mu t +\sigma W_t
    \end{equation}
for positive constants $\mu$ and $\sigma$, then the signal-to-noise ratio is $\rho:=\frac{2\mu}{\sigma}$. By Brownian scaling, the process
\[\tilde X_t:=\frac{\mu}{\sigma^2} X_{\sigma^2t/\mu^2}\]
then satisfies 
\[\tilde X_t=\theta t + \tilde W_t,\]
where $\tilde W_t:=\frac{\mu}{\sigma} W_{\sigma^2t/\mu^2}$ is a standard Brownian motion. Consequently, the fixed sample size $T^\rho_{\alpha}$ needed to achieve 
a certain precision $1-\alpha$ with signal-to-noise ratio  $\rho$ satisfies
\[T^{\rho}_{\alpha}=\frac{4}{\rho^2}T_\alpha.\]
Similarly, the optimal stopping time $\tau^\rho_\alpha$ needed for precision $1-\alpha$ in the setting of \eqref{Xmu}  
satisfies in expectation
\[\E_a[\tau^\rho_{\alpha}]=\frac{4}{\rho^2}
E_a[\tau_\alpha].\]
In conclusion, both the fixed sample size and the average sequential sample size scale inversely with the square of the signal-to-noise ratio. Therefore, the relative sample size reduction
(defined in \eqref{reduction} below) is independent of the signal-to-noise ratio, and we thus use $\mu=\sigma=1$ throughout the article.

\section{Relative efficiency}\label{sec3}

In this section we study the 
relative efficiency
\[
    \frac{\E_{1}[\tau_\alpha]}{T_\alpha}
\]
of the sequential probability ratio test compared with tests with fixed sample size, where $\alpha\in(0,\frac{1}{2})$ is a given acceptable error probability.
To do this, we introduce the
function $f: (0,\frac{1}{2})\to [0,1]$ given by 
\begin{equation} 
\label{ffunction}
f(\alpha):=\frac{\E_1[\tau_\alpha]}{T_\alpha}=
\frac{\eta(\alpha)}{\left(\Phi^{-1}(\alpha)\right)^2},
\end{equation}
i.e. the quotient of the average sample size using sequential methods and the corresponding fixed sample size with the same power $1-\alpha$. 
We also note that 
\begin{equation}
    \label{reduction}
    1-f(\alpha)=\frac{T_\alpha-\E_{1}[\tau_\alpha]}{T_\alpha}
\end{equation}
is the average sample size reduction.

The following is our main result.
 
\begin{theorem}\label{main}
The function $f$ is increasing on $(0,\frac{1}{2})$, with $f(0+)=\frac{1}{4}$ and $f(\frac{1}{2}-)=\frac{2}{\pi}$.
\end{theorem}

\begin{remark}
In Figure~\ref{fig1}, the function $f$ is plotted.
Note that it follows from Theorem~\ref{main} that the average sample size reduction $1-f(\alpha)$ (compare \eqref{reduction}) is increasing in the power $1-\alpha$ of the test. Moreover, 
     $1-f(0+)=\frac{3}{4}$, so the maximal reduction in average sample size is 75\%. Similarly, $1-f(\frac{1}{2})=1-\frac{2}{\pi}\approx 0.36$, so the minimal reduction is 36\%. The reduction for a few common values of $\alpha$ is listed in Table~\ref{tab}. 
\end{remark}

\begin{figure}[!htbp]
\centering
\begin{minipage}{0.5\textwidth}
    \centering
    \includegraphics[width=\textwidth]{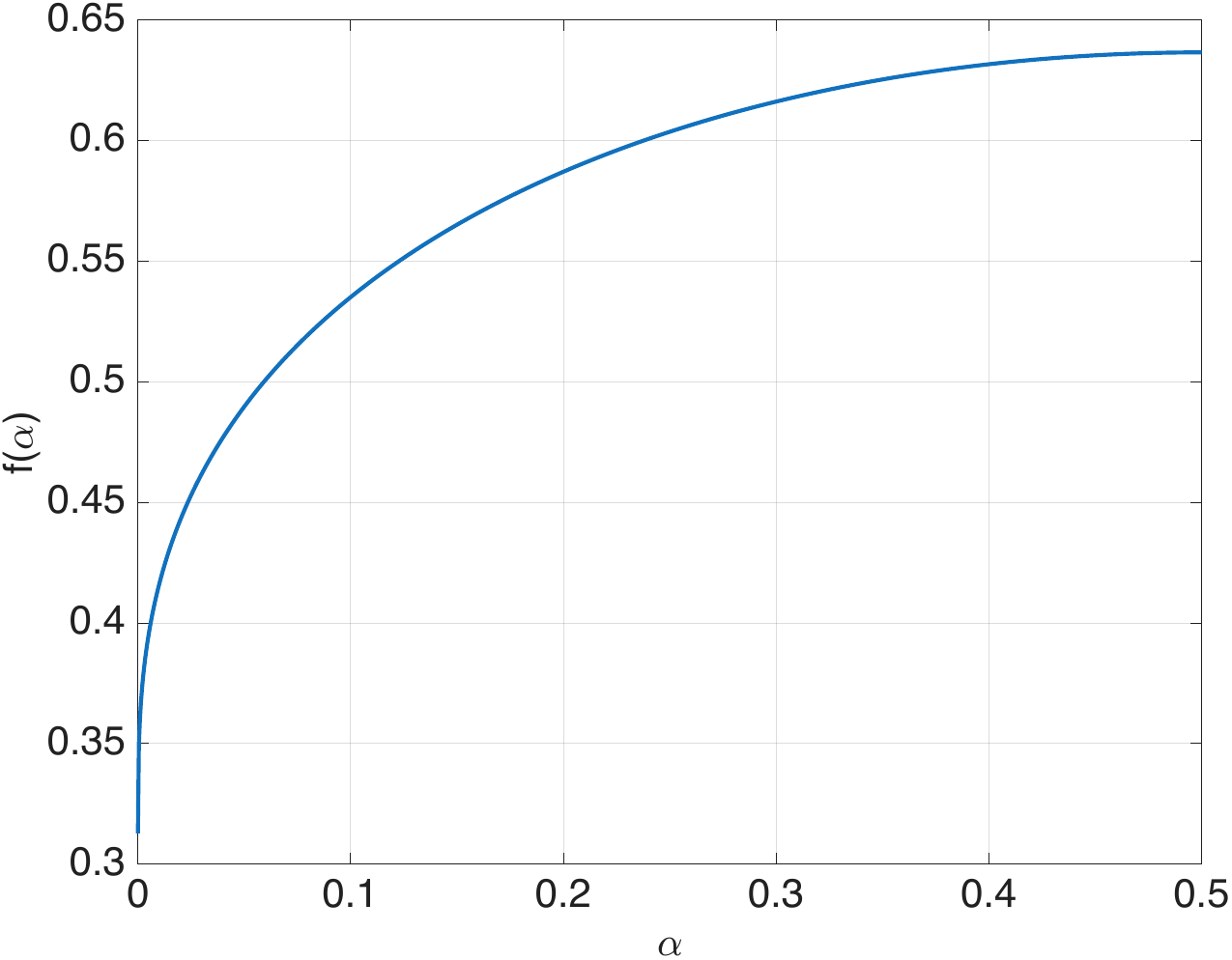}
\end{minipage}\hfill
\begin{minipage}{0.5\textwidth}
    \centering
    \includegraphics[width=\textwidth]{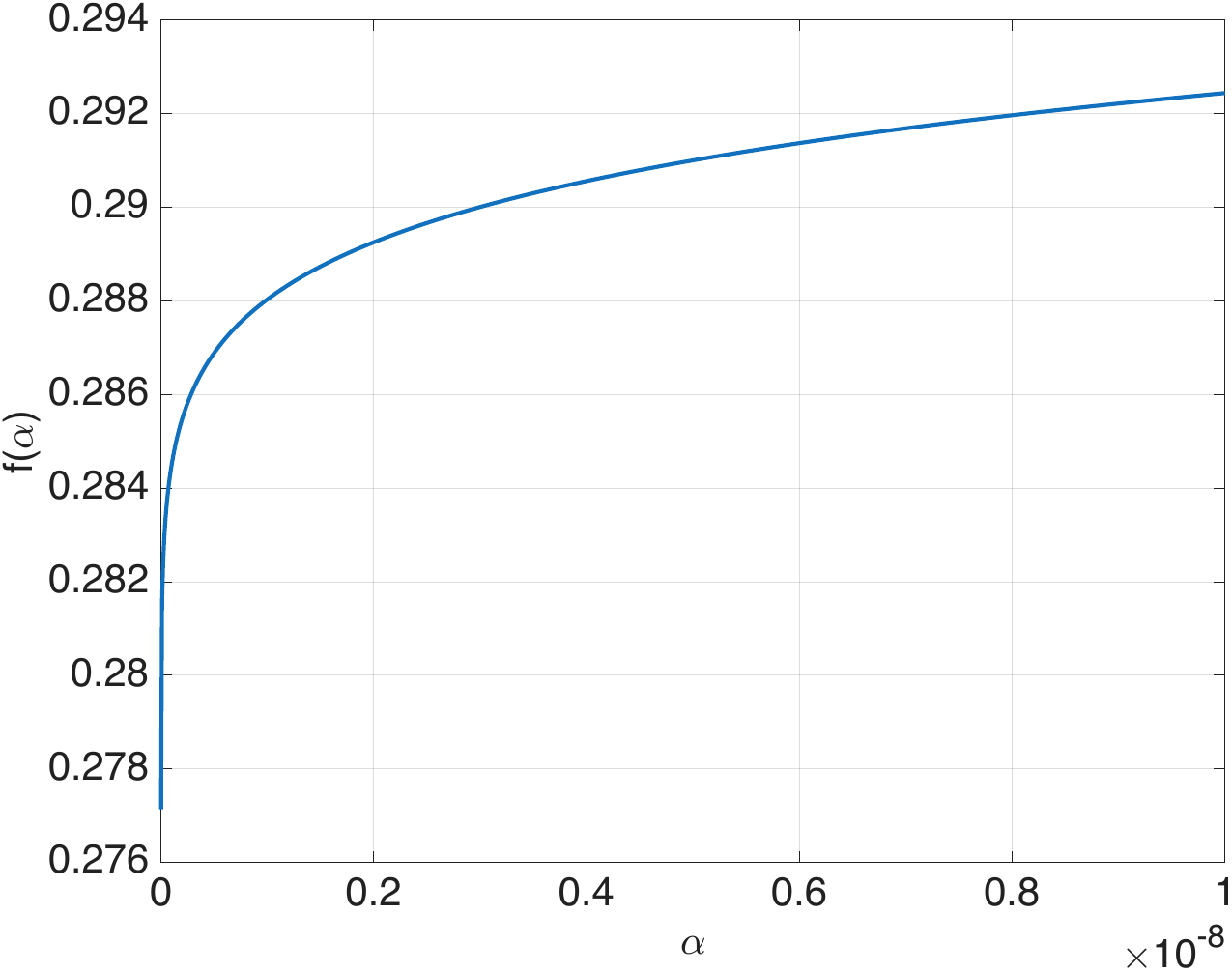}
\end{minipage}
\caption{The function $f(\alpha)$ for $\alpha\in(0,\frac{1}{2})$ (left), and $f(\alpha)$ for $\alpha\in(0,10^{-8})$ (right).}
\label{fig1}
\end{figure}

\begin{table}[!htbp]
\centering
\caption{Relative efficiency of SPRT.}
\begin{tabular}{c|cc}
\hline\hline
Power $1-\alpha$ & $f(\alpha)$ & Average reduction \\
\hline
0.80 & 0.5871 & 41.29\% \\
0.90 & 0.5351 & 46.49\% \\
0.95 & 0.4897 & 51.03\% \\
0.99 & 0.4160 & 58.40\% \\
0.999&0.3609&63.91\%\\
\hline
\end{tabular}\label{tab}
\end{table}

\begin{remark}
The limit $f(0+)=\frac{1}{4}$ was derived in \cite{A}, see also \cite[Section 6]{EG} and \cite[p.193]{S}. 
\end{remark}

For the proof of Theorem~\ref{main}, we will need the following lemma.

\begin{lemma}\label{lem}
The inequality 
\begin{equation}\label{ineqMills}
z(2+z^2)\Phi(z)(1-\Phi(z))\geq (1+z^2)(2\Phi(z)-1)\varphi(z)
\end{equation}
holds for all $z\geq 0$.
\end{lemma}

\begin{remark}
Note that Lemma~\ref{lem} provides a lower bound for the Mills ratio \mbox{$M(z)=\frac{1-\Phi(z)}{\varphi(z)}$}. While bounds on the Mills ratio have been studied by many authors (see, e.g., \cite{Mills_ratio_ineq} and \cite{GU}),
we have not been able to find \eqref{ineqMills} in the literature.
\end{remark}

\begin{proof}
Define $g:[0,\infty)\to\R$ by 
\[g(z)=z(2+z^2)\Phi(z)(1-\Phi(z))-(1+z^2)(2\Phi(z)-1)\varphi(z),\]
and note that $g(0)=g(\infty)=0$. Moreover, straightforward differentiation gives
\begin{eqnarray}\label{gprime}
     g'(z) &=& (2+3z^2)\Phi(z)(1-\Phi(z))- 
    3z(2\Phi(z)-1)\varphi(z)-2(1+z^2)\varphi^2(z)\\
\notag     &=& \frac{2+3z^2}{z(2+z^2)}g(z)+\frac{2-z^2}{z(2+z^2)}(2\Phi(z)-1)\varphi(z) -2(1+z^2)\varphi^2(z).
 \end{eqnarray}

We now claim that 
\begin{equation}
    \label{ineq}
    \frac{2-z^2}{z(2+z^2)}(2\Phi(z)-1)-2(1+z^2)\varphi(z)\leq 0
\end{equation}
for $z>0$.
To see that \eqref{ineq} holds, first note that it clearly holds for $z\geq \sqrt 2$ since both terms then are negative. Next, let 
\[h(z):=2\Phi(z)-1-z(1+z^2)(2+z^2)\varphi(z),\]
and note that it suffices to check that $h(z)\leq 0$ for $0\leq z\leq \sqrt 2$.
Clearly, $h(0)=0$ and 
\[h'(z)=z^2(z^4-2z^2-7)\varphi(z),
\]
which is negative on $(0,\sqrt{1+\sqrt 8})\supseteq(0,\sqrt 2)$.
Therefore, $h(z)\leq 0$ for $z\in (0,\sqrt 2)$, so
\eqref{ineq} holds.

Combining \eqref{gprime} and \eqref{ineq}, we find that
\begin{equation}\label{gdiff}
g'(z)\leq \frac{2+3z^2}{z(2+z^2)}g(z)
\end{equation}
for $z>0$.
Together with $g(0)=g(\infty)=0$, this implies that $g(z)\geq 0$ for all $z\geq 0$. Indeed, assuming the existence of an $z_0\in(0,\infty)$ s.t. $g(z_0)<0$, the inequality \eqref{gdiff} implies that $g(z)\leq g(z_0)<0$ for all $z\geq z_0$. This contradicts $g(\infty)=0$, finishing the proof.
\end{proof}

\noindent
{\em Proof of Theorem~\ref{main}.}
We first derive the limits of $f(\alpha)$ as $\alpha\downarrow 0$ and $\alpha\uparrow 1/2$. Using l'H\^opital's rule twice, we find that 
\begin{eqnarray*}
    f(0+) &=& \lim_{\alpha\to 0}\frac{\eta''(\alpha)}{2\left(\frac{1}{\varphi^2(\Phi^{-1}(\alpha))}+\frac{\left(\Phi^{-1}(\alpha)\right)^2}{\varphi^2(\Phi^{-1}(\alpha))}\right)}\\
    &=& \lim_{z\to\infty}\frac{\varphi^2(z)}{4 \Phi^2(z)(1-\Phi(z))^2(1+z^2)} =\frac{1}{4}
\end{eqnarray*}
since  $\lim_{z\to\infty}\frac{\varphi(z)}{z(1-\Phi(z))}=1$.
Similarly, 
\begin{eqnarray*}
    f\left(\frac{1}{2}-\right) &=& \lim_{\alpha\to \frac{1}{2}}\frac{\eta''(\alpha)}{2\left(\frac{1}{\varphi^2(\Phi^{-1}(\alpha))}+\frac{\left(\Phi^{-1}(\alpha)\right)^2}{\varphi^2(\Phi^{-1}(\alpha))}\right)}\\
    &=& \lim_{\alpha\to \frac{1}{2}}\frac{\varphi^2(\Phi^{-1}(\alpha))}{4\alpha^2(1-\alpha)^2\left(1+\left(\Phi^{-1}(\alpha)\right)^2\right)}\\
    &=& 4\varphi^2(0) =\frac{2}{\pi}.
    \end{eqnarray*}
    
We now investigate the monotonicity of $f$.
Straightforward differentiation shows that
\[    f'(\alpha)=\frac{\eta'(\alpha)}{(\Phi^{-1}(\alpha))^2}- \frac{2\eta(\alpha)}{(\Phi^{-1}(\alpha))^3\varphi(\Phi^{-1}(\alpha))}.\]
To prove that $f'(\alpha)\geq 0$ for $\alpha\in(0,1/2)$, define
\begin{eqnarray*}   
  g(\alpha)&:=&2f'(\alpha)(\Phi^{-1}(\alpha))^3\varphi(\Phi^{-1}(\alpha))\\
    &=& 2\eta'(\alpha)\Phi^{-1}(\alpha)\varphi\left(\Phi^{-1}(\alpha)\right)-4\eta(\alpha),
    \end{eqnarray*}
and note that it suffices to show that $g\leq 0$. Since
$g(1/2)=0$, for this it suffices to check that $g'(\alpha)\geq 0$.
Another differentiation yields
\begin{eqnarray*}
    g'(\alpha) &=& 2\eta''(\alpha)\Phi^{-1}(\alpha)\varphi(\Phi^{-1}(\alpha)) -2\left(1+(\Phi^{-1}(\alpha))^2\right) \eta'(\alpha)\\
    &=& \frac{\Phi^{-1}(\alpha)\varphi(\Phi^{-1}(\alpha))}{\alpha^2(1-\alpha)^2 }-2\left(1+(\Phi^{-1}(\alpha))^2\right) \eta'(\alpha),
\end{eqnarray*}
where the second equality follows from the fact that 
\[\alpha^2(1-\alpha)^2\eta''(\alpha)=\frac{1}{2}.\]

Defining
\[
    h(\alpha) := \frac{g'(\alpha)}{1+(\Phi^{-1}(\alpha))^2}\\
    = -2\eta'(\alpha)+\frac{\Phi^{-1}(\alpha)\varphi(\Phi^{-1}(\alpha))}{\alpha^2(1-\alpha)^2(1+(\Phi^{-1}(\alpha))^2)} 
\]
we clearly have $h(1/2)=0$. Consequently, it suffices to show that $h'(\alpha)\leq 0$. Differentiation gives
\[h'(\alpha)= \frac{-2\Phi^{-1}(\alpha)\left(
\Phi^{-1}(\alpha)\frac{2+(\Phi^{-1}(\alpha))^2}{1+(\Phi^{-1}(\alpha))^2}
+\frac{(1-2\alpha)\varphi(\Phi^{-1}(\alpha))}{\alpha(1-\alpha)}
\right)}{\alpha^2(1-\alpha)^2\left(1+(\Phi^{-1}(\alpha))^2\right)},\] 
so it suffices to check that 
\[\Phi^{-1}(\alpha)\frac{2+(\Phi^{-1}(\alpha))^2}{1+(\Phi^{-1}(\alpha))^2}
+\frac{(1-2\alpha)\varphi(\Phi^{-1}(\alpha))}{\alpha(1-\alpha)}\leq 0\]
for $\alpha\in(0,1/2)$. 
Writing $z=-\Phi^{-1}(\alpha)$, however, this follows immediately from Lemma~\ref{lem}.
\qed
\newline

Although the maximal expected reduction is up to $75\%$, the convergence to this limit is rather slow, as indicated by Figure~\ref{fig1} (and also noted in \cite{EG}). Below, we provide the exact convergence rate, which is of order $\frac{\ln(-\ln(\alpha))}{\ln(\alpha)} $. Thus, a reduction close to the theoretical limit only happens when one requires extremely high precision. For example, to reach a $70\%$, $72\%$ or $74\%$ sample size reduction, an error probability $\alpha$ of approximately $2\times 10^{-7}$, $2\times 10^{-12}$ and $2\times 10^{-40}$ is required, respectively.

\begin{proposition}
We have
\begin{equation}\label{asymp}
f(\alpha) = \frac{1}{4}- \frac{\ln(-\ln\alpha)}{8\ln\alpha}
        +o\left(\frac{\ln(-\ln \alpha)}{\ln\alpha}\right)
\end{equation}
as $\alpha\to 0$.
\end{proposition}

\begin{proof}
Using 
\[2\eta(\alpha)+\ln\alpha=2\alpha\ln\frac{\alpha}{1-\alpha}+\ln(1-\alpha),\]
one sees that 
\begin{equation}\label{O}
    2\eta(\alpha)=-\ln \alpha   +o(1)
    \end{equation}
as $\alpha\to 0$.
We also claim that 
\begin{equation}
    \label{claim}
    (\Phi^{-1}(\alpha))^2 =-2\ln \alpha-\ln(-4\pi\ln\alpha)+o(1).
\end{equation}
Combining \eqref{O} and \eqref{claim}, \eqref{asymp} follows. 

To prove \eqref{claim}, we first perform a change of variable by letting $z:= -\Phi^{-1}(\alpha)$, and note that $\Phi(-z)=\alpha$ and $z\to\infty$ as $\alpha\to 0$. 
By classical estimates of the Mills ratio, we have
\[ \frac{\varphi(z)z^2}{z(1+z^2)}\leq\alpha\leq \frac{\varphi(z)}{z}.\]
Taking logarithm on both sides, we find that
\[z^2 + \ln (2\pi z^2) \leq -2\ln \alpha \leq z^2 + \ln (2\pi z^2)+2 \ln(1+z^{-2}), \]
which yields 
\begin{equation}\label{ineqs}
-\frac{1}{z^2}\leq z^2+2\ln\alpha +\ln(2\pi z^2) \leq 0.\end{equation}
Using $z^2\leq -2\ln\alpha$, we find that 
\begin{equation}
    \label{ineq1}
    -\frac{1}{z^2}\leq z^2+2\ln\alpha +\ln(-4\pi \ln\alpha).
    \end{equation}
Also, by \eqref{ineqs}, for any $\ep>0$ we have $(1+\ep) z^2\geq -2\ln\alpha$ for $z$ large enough, so 
\begin{equation}\label{ineq2}
z^2+2\ln\alpha +\ln(-4\pi \ln\alpha)\leq \ln(1+\ep)\leq \ep.
\end{equation}
It follows from \eqref{ineq1} and \eqref{ineq2} that 
\[z^2+2\ln\alpha +\ln(-4\pi \ln\alpha)=o(1),\]
so \eqref{claim} holds. This finishes the proof.
\end{proof}

\begin{remark}
One can also check that $f$ approaches its limit at $\alpha = \frac{1}{2}$ quadratically, i.e. $f(\alpha) = \frac{2}{\pi}+\mathcal{O}(( \frac{1}{2}-\alpha)^2)$
    as $\alpha\to\frac{1}{2}$.
\end{remark}

\section{The asymmetric case}\label{sec4}

Above we have treated the symmetric case, where both error probabilities $\P_{-1}(d=1)$ and $\P_1(d=-1)$ are bounded by the same constant $\alpha$.
If, instead, two possibly different bounds $\alpha$ and $\beta$ are provided for the two types of error, then one requires that
\begin{equation}\label{asym}
\P_{-1}(d=1)\leq \alpha \quad\quad\& \quad\quad \P_{1}(d=-1)\leq \beta.
\end{equation}
In this section we extend our study to this case.

First, it is straightforward to check that the smallest possible deterministic time $T_{\alpha,\beta}$ for which a test satisfying \eqref{asym} exists is given by
\[T_{\alpha,\beta}=\frac{(\Phi^{-1}(\alpha)+ \Phi^{-1}(\beta))^2}{4},\]
with corresponding decision variable
\[d=
1_{\{X_{T_{\alpha,\beta}} \geq T_{\alpha,\beta} + \sqrt{T_{\alpha,\beta}}\Phi^{-1}(\beta)\}}-
1_{\{X_{T_{\alpha,\beta}} < T_{\alpha,\beta}+ \sqrt{T_{\alpha,\beta}}\Phi^{-1}(\beta)\}}.\]
On the other hand, the stopping time in the SPRT is given by
\[\tau_{\alpha,\beta}: = \inf\left\{t\geq 0: X_t\notin \left(\frac{1}{2}\ln\frac{1-\beta}{\alpha},\frac{1}{2}\ln\frac{\beta}{1-\alpha}\right)\right\},\]
with decision variable 
\[d=
1_{\{X_{\tau_{\alpha,\beta}}=\frac{1}{2}\ln\frac{\beta}{1-\alpha}\}}
-1_{\{X_{\tau_{\alpha,\beta}}=\frac{1}{2}\ln\frac{1-\beta}{\alpha}\}}. \]
In addition, $\E_{-1}\left[\tau_{\alpha,\beta}\right]=\omega(\alpha, \beta)$ and $\E_{1}\left[\tau_{\alpha,\beta}\right]=\omega(\beta, \alpha)$, where 
\[\omega(\alpha,\beta):= \frac{1}{2}\left(\alpha\ln\frac{\alpha}{1-\beta}+(1-\alpha)\ln\frac{1-\alpha}{\beta}\right)\,.\]
We define 
\[F(\alpha,\beta):=\frac{\max\{\E_{-1}\left[\tau_{\alpha,\beta}\right],\E_{1}\left[\tau_{\alpha,\beta}\right]\}}{T_{\alpha,\beta}},\]
so that $1-F(\alpha,\beta)$ denotes the 
average sample size reduction under the 
hypothesis with the largest expected sample size.

The following result shows that $F$ 
attains its smallest values along the diagonal $\alpha=\beta$,
so the reduction in average sample size for the asymmetric problem is bounded by the reduction in the symmetric case.
In particular, the average sample size reduction is always smaller than 75\%. For a graphical illustration of the function $F$, see Figure~\ref{fig2}.

\begin{theorem}
\label{mainasym}
For $(\alpha,\beta)\in(0,\frac{1}{2})^2$ we have that
\begin{equation}\label{ineqasym}
F(\alpha,\beta)\geq f(\min\{\alpha,\beta\})\,.
\end{equation}
\end{theorem}

\begin{figure}[!htbp]
\centering
\begin{minipage}{0.5\textwidth}
    \centering
    \includegraphics[width=\textwidth]{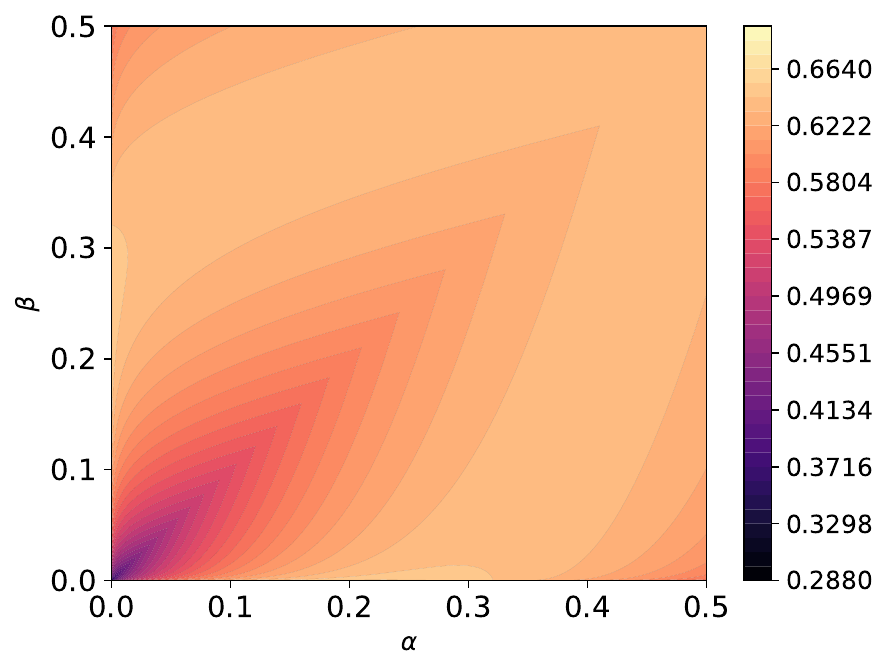}
\end{minipage}\hfill
\begin{minipage}{0.5\textwidth}
    \centering
    \includegraphics[width=\textwidth]{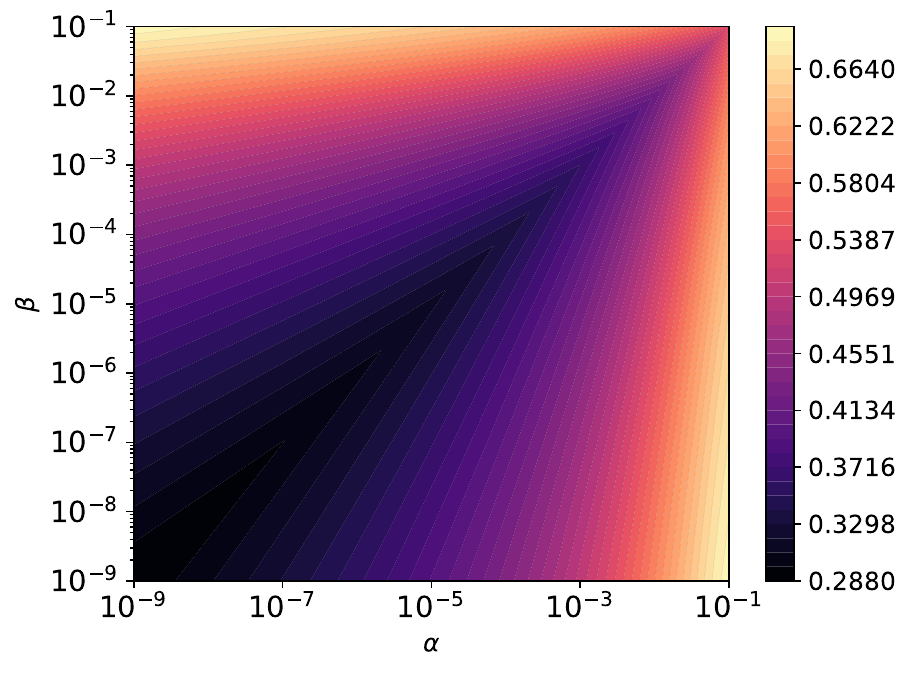}
\end{minipage}
\caption{The function $F(\alpha,\beta)$ on $(0,\frac{1}{2})^2$ (left), and $F(\alpha,\beta)$ for $(\alpha,\beta)\in(0,\frac{1}{10})^2$ with logarithmic axes (right).}
\label{fig2}
\end{figure}

In the proof of Theorem~\ref{mainasym} we will use the following auxiliary results.

\begin{lemma}
\label{lemasym}
    We have 
    \[\max\{\omega(\alpha,\beta), \omega(\beta,\alpha)\}
=\omega(\max\{\alpha,\beta\},\min\{\alpha,\beta\}).\]
\end{lemma}

\begin{proof}
Defining $g:[\beta,\frac{1}{2}]\to\R$ by 
\begin{eqnarray*}
g(\alpha) &:=& 2\omega(\alpha,\beta)-2\omega(\beta,\alpha)\\
&=& (1+\alpha-\beta)\ln\frac{\alpha}{1-\beta} + (1-\alpha+\beta)\ln\frac{1-\alpha}{\beta},
\end{eqnarray*}
it is straightforward to check that $g$ is concave with $g(\beta)=0$
and $g(\frac{1}{2})\geq 0$. Therefore $g\geq 0$, and the result follows.
\end{proof}

\begin{lemma}\label{lem:2d_ineq}
We have that 
    \begin{equation}\label{ineq3}
    \left(\Phi^{-1}(\beta)\right)^2\varphi^3(\Phi^{-1}(\alpha))\frac{2\alpha-1}{\alpha^2(1-\alpha)^2}+2(1-2\beta)\ln\frac{\beta}{1-\beta}\Phi^{-1}(\alpha)\geq 0
\end{equation}
for $0< \beta\leq \alpha <\frac{1}{2}$.
\end{lemma}

\begin{proof}
Denote by $z = -\Phi^{-1}(\alpha)> 0$. Rearranging the terms, we have that \eqref{ineq3} is equivalent to
\begin{equation}\label{equiv}
\varphi^3(z)\frac{1-2\alpha}{\alpha^2(1-\alpha)^2}\leq 4f(\beta)z,
\end{equation}
    where $f$ is as in \eqref{ffunction}. Observe that 
\begin{equation}\label{est1}1-2\alpha=2\Phi(z)-1=2\int_0^{z} \varphi(x)dx \leq  \sqrt{\frac{2}{\pi}}z
\end{equation}
since the Gaussian density is decreasing on $[0,\infty)$, and that 
    \begin{eqnarray*}
    \alpha(1-\alpha) &=& \frac{1-(1-2\alpha)^2}{4} = 
    \frac{1- (2\Phi(z)-1)^2}{4}.
    \end{eqnarray*}

Now recall the classical estimate (see \cite{P} and \cite{Williams})
\begin{eqnarray} \label{est}
\left(2\Phi(z)-1\right)^2 &=& \int_{-z}^{z}\int_{-z}^{z}\varphi(x)\varphi(y)\,dx\,dy \\
\notag
&\leq& \iint_{B(\frac{2z}{\sqrt\pi})}\varphi(x)\varphi(y)\,dx\,dy=
1-e^{-\frac{2z^2}{\pi}}
\end{eqnarray}
for $z\geq 0$, where $B(\frac{2z}{\sqrt\pi})$ is the disc with area $4z^2$ centered at the origin.
  Using \eqref{est}, we have
    \[\frac{1}{\alpha^2(1-\alpha)^2}\leq \frac{16}{(1-(1-e^{-\frac{2z^2}{\pi}}))^2} = 16 e^{\frac{4z^2}{\pi}}.\]
Consequently, also using \eqref{est1}, 
    \[\varphi^3(z)\frac{1-2\alpha}{\alpha^2(1-\alpha)^2} \leq  16\sqrt{\frac{2}{\pi}}z e^{\frac{4z^2}{\pi}}  \varphi^3(z)= \frac{8}{\pi^2} z e^{(\frac{4}{\pi}-\frac{3}{2}) z^2}
        \leq z\leq 4f(\beta)z\]
    since $f(\beta) \geq \frac{1}{4}$ by Theorem~\ref{main}. This proves that \eqref{equiv} holds, which finishes the proof.
\end{proof}

\noindent
{\em Proof of Theorem~\ref{mainasym}.}
By symmetry, it suffices to prove \eqref{ineqasym} when   $0<\beta\leq\alpha<1/2$, and in that case, it follows from Lemma~\ref{lemasym} that
\[F(\alpha,\beta) = \frac{\omega(\alpha,\beta)}{T_{\alpha,\beta}}= 
   \frac{2\left(\alpha\ln\frac{\alpha}{1-\beta}+(1-\alpha)\ln\frac{1-\alpha}{\beta}\right)}{\left(\Phi^{-1}(\alpha)+\Phi^{-1}(\beta)\right)^2}.\]
Define for fixed $\beta\in(0,\frac{1}{2})$ the function 
$g:[\beta,\frac{1}{2})\to\mathbb R$ by
\[g(\alpha):= 4(\Phi^{-1}(\beta))^2\left(\alpha\ln\frac{\alpha}{1-\beta}+(1-\alpha)\ln\frac{1-\alpha}{\beta}\right) +(1-2\beta)\ln\frac{\beta}{1-\beta}\left(\Phi^{-1}(\alpha)+\Phi^{-1}(\beta)\right)^2.\]
To prove \eqref{ineqasym}, it then suffices to check that $g(\alpha)\geq 0$. Clearly, $g(\beta)=0$. Evaluating $g$ at $\frac{1}{2}$, we have
\[g\left(\frac{1}{2}\right) = \left(\Phi^{-1}(\beta)\right)^2\left(2 \ln \frac{1}{4\beta(1-\beta)} +(1-2\beta )\ln\frac{\beta}{1-\beta}\right)=:\left(\Phi^{-1}(\beta)\right)^2\ h(\beta).\]
Observe that 
\[h'(\beta) = \frac{2\beta(\beta-1) \ln\frac{\beta}{1-\beta} +2\beta-1}{\beta(1-\beta)}\] 
so that $h'(\frac{1}{2}) = 0$. Also, 
$h''(\beta) = \left(\frac{1-2\beta}{\beta(1-\beta)}\right)^2 >0$, so $h'(\beta)$ is strictly negative on $(0,\frac{1}{2})$. Therefore the function $h$ is strictly decreasing on $(0,\frac{1}{2})$, and 
\[g\left(\frac{1}{2}\right) =\left(\Phi^{-1}(\beta)\right)^2\ h(\beta)\geq \left(\Phi^{-1}(\beta)\right)^2 h(\frac{1}{2}) = 0.\] 
Next, straightforward differentiation shows that 
\[
g'(\alpha) = 4(\Phi^{-1}(\beta))^2\left(\ln\frac{\alpha}{1-\alpha}+\ln\frac{\beta}{1-\beta}\right)
+2(1-2\beta)\frac{\Phi^{-1}(\alpha)+\Phi^{-1}(\beta)}{\varphi(\Phi^{-1}(\alpha))}\ln\frac{\beta}{1-\beta}\,.
\]
Denoting $z = -\Phi^{-1}(\beta)$, it follows that
\begin{align*}
 g'(\beta)&= 4\Phi^{-1}(\beta)\left(2\Phi^{-1}(\beta)+\frac{1-2\beta}{\varphi(\Phi^{-1}(\beta))}\right)\ln\frac{\beta}{1-\beta}\\
 & = -4z\left(-2z+\frac{2 \int_0^{z}\varphi(u)du}{\varphi(z)}\right)\ln\frac{\beta}{1-\beta}\\
 &>  -4z\left(-2z+\frac{2 \int_0^{z }\varphi(z)du}{\varphi(z)}\right)\ln\frac{\beta}{1-\beta} = 0.
\end{align*}
Similarly,
\begin{align*}
    g'\left(\frac{1}{2}-\right)&= 2\Phi^{-1}(\beta)\left(2\Phi^{-1}(\beta)+(1-2\beta)\sqrt{2\pi}\right)\ln\frac{\beta}{1-\beta}\\
    &<-2z\left(-2z+2\sqrt{2\pi}\int_0^{z }\varphi(0)du\right)\ln\frac{\beta}{1-\beta} = 0.
\end{align*}
Hence, by continuity, a root of $g'$ exists in $(\beta,\frac{1}{2})$. In the remainder, we will check that $g'$ is convex. Together with $g'\left(\frac{1}{2}-\right)<0$, this implies uniqueness of such root. Combined with the fact that $g(\beta) = 0$, $g'(\beta)>0$, and $g(\frac{1}{2}-)>0$, we find $g(\alpha)\geq 0$ for all $0<\beta\leq \alpha<\frac{1}{2}$. 

To show that $g'$ is convex, we compute
\begin{eqnarray*} 
g'''(\alpha)&=& 4(\Phi^{-1}(\beta))^2\frac{2\alpha-1}{\alpha^2(1-\alpha)^2}\\&& \hspace{-15mm}+\frac{2(1-2\beta)}{\varphi^3(\Phi^{-1}(\alpha))}\left(\Phi^{-1}(\beta)+2\Phi^{-1}(\alpha)(2+\Phi^{-1}(\alpha)(\Phi^{-1}(\alpha)+\Phi^{-1}(\beta)))\right)\ln\frac{\beta}{1-\beta}\\
&\geq& 4(\Phi^{-1}(\beta))^2\frac{2\alpha-1}{\alpha^2(1-\alpha)^2}+4\frac{2(1-2\beta)\Phi^{-1}(\alpha)}{\varphi^3(\Phi^{-1}(\alpha))}\ln\frac{\beta}{1-\beta}\,,
\end{eqnarray*} which is positive by Lemma \ref{lem:2d_ineq}, concluding the proof.
\qed

\begin{remark}
    It is a consequence of Theorems~\ref{main} and \ref{mainasym} that $F(\alpha,\beta)\geq \frac{1}{4}$.
In contrast to the symmetric case of Section~\ref{sec3}, however, there is no (non-trivial) upper bound on $F$. In fact, it is straightforward to check that  
    \[\lim_{\beta\to 0}F(\alpha,\beta)=1-\alpha,\]
and it follows that
    \[\liminf_{(\alpha,\beta)\to(0,0)}F(\alpha,\beta)=\frac{1}{4}<1=\limsup_{(\alpha,\beta)\to(0,0)}F(\alpha,\beta).\]
One may also note that this contradicts the statement from \cite{A} that $F(\alpha,\beta)\leq 17/30$ for $\alpha,\beta\in(0,0.03)$.
\end{remark}

\end{document}